\documentclass[letterpaper]{article} 
\usepackage{ijcai19}  
\usepackage{times}  
\usepackage{helvet}  
\usepackage{courier}  
\usepackage{url}  
\usepackage{graphicx}  
\usepackage{wrapfig}
\usepackage[utf8]{inputenc} 
\usepackage[T1]{fontenc}    
\usepackage{url}            
\usepackage{booktabs}       
\usepackage{amsfonts}       
\usepackage{nicefrac}       
\usepackage{microtype}      
\usepackage{adjustbox}
\usepackage[ruled]{algorithm2e}
\usepackage{color,soul}
\usepackage[english]{babel}
\usepackage{amsmath, amsthm}
\usepackage{amssymb}
\usepackage{tabularx}
\usepackage{verbatim}
\usepackage[small]{caption}

\theoremstyle{plain}
\newtheorem{theorem}{Theorem}
\newtheorem{lemma}[theorem]{Lemma}

\theoremstyle{definition}
\newtheorem{definition}{Definition}

\newtheorem{assumption}{Assumption}

\DeclareMathOperator{\tr}{tr}
\DeclareMathOperator{\ent}{\mathcal{S}}
\DeclareMathOperator{\R}{\mathbb{R}}
\DeclareMathOperator{\rank}{\mathrm{rank}\,}

\DeclareMathOperator*{\argmax}{argmax}
\newcommand{\schatteneps}[2]{\|#1\|_{S_{#2,\varepsilon}}}

\newcommand{\bS}{\mathbb{S}}
\newcommand{\bR}{\mathbb{R}}
\newcommand{\cmmnt}[1]{}

\usepackage{tikz}
\usepackage{pgfplots}
\pgfplotsset{width=11cm,height=7cm}

\title{Entropy-Penalized Semidefinite Programming}

\author{
Mikhail Krechetov$^1$\footnote{Contact Author}\and
Jakub Mare\v{c}ek$^2$\and 
Yury Maximov$^{1,3}$\And
Martin Tak\'a\v{c}$^4$\\
\affiliations
$^1$ Skolkovo Institute of Science and Technology, Nobel Street 3, Moscow, Russia \\
$^2$ IBM Research -- Ireland, Technology Campus Damastown, Dublin D15, Ireland\\
$^3$ Los Alamos National Laboratory, MS-B284, Los Alamos, NM 87545, USA \\ 
$^4$ Lehigh University, 200 West Packer Avenue, Bethlehem, PA 18015, USA
\emails
mikhail.krechetov@skoltech.ru,
jakub.marecek@ie.ibm.com,
yury@lanl.gov,
mat614@lehigh.edu
}

\begin{document}
\maketitle
\begin{abstract}
Low-rank methods for semi-definite programming (SDP) have gained a lot of interest recently, especially in machine learning applications. Their analysis often involves determinant-based or Schatten-norm penalties, which are difficult to implement in practice due to high computational efforts. In this paper, we propose Entropy-Penalized Semi-Definite Programming (EP-SDP)\footnote{See  https://github.com/mkrechetov/epsdp for the implementation details.}\!\!, which provides a unified framework for a broad class of penalty functions used in practice to promote a low-rank solution. We show that EP-SDP problems admit an efficient numerical algorithm, having (almost) linear time complexity of the gradient computation; this makes it useful for many machine learning and optimization problems. We illustrate the practical efficiency of our approach on several combinatorial optimization and machine learning problems.

%
%
\end{abstract}

\section{Introduction}

Semidefinite programming (SDP) has become a key tool in solving numerous problems across operations research, machine learning, and artificial intelligence. While there are too many applications of SDP to present even a representative sample, inference in graphical models \cite{ExpFamilies,erdogdu2017inference}, multi-camera computer vision \cite{torr2003solving}, 
and applications of polynomial optimization \cite{Parrilo,lasserre2015introduction}   
in power systems \cite{7024950}
stand out. Under some assumptions \cite{Madani2014}, the rank at the optimum of the SDP relaxation is bounded from above by the tree-width of a certain hypergraph, plus one. When a rank-one solution is not available, it is often not needed \cite{Marecek2017}, as one should like to construct a stronger SDP relaxation.


%
%
Penalization of the objective is a popular approach for obtaining low-rank solutions, at least in theory  \cite{recht2010guaranteed,Lemon,zhou19dis,Fawzi2019}. 
Notice that without a further penalization, an interior-point method for SDP provides a solution on the boundary of the feasible set, where SDP corresponds to the optimum of the highest rank, whenever there are optima of multiple ranks available. 
The use of a penalization provides a counter-balance in this respect. 
In practice, however, the penalties are often ignored, as it is believed that their computation is too demanding for large-scale problems and does not guarantee low-rank solutions in general.


An alternative approach develops numerical optimization methods that seek \emph{a priori} low-rank solutions. 
This approach, widely attributed to \cite{BurerMonteiro}, considers a factorization of a semidefinite matricial variable $X = V\cdot V^\top$ with $V\in \R^{n\times k}$ for increasing $1 \le k \ll n$. In general, the resulting problems are  non-convex.
Early analyses required determinant-based penalty terms \cite{burer2005local}, although no efficient implementations were known.
Under mild assumptions, for large-enough $k$, 
there is a unique optimum over such a factorization even without a penalization and it recovers the optimum of the initial SDP problem \cite{boumal2016non}.
For smaller values of $k$, it is known that 
the low-rank relaxation achieves $\mathcal O(1/k)$ relative error \cite{Montanari}. Much more elaborate analyses \cite{erdogdu2018convergence} are now available.
Especially when combined with efficient gradient computation, e.g. within low-rank coordinate  descent \cite[e.g.]{Marecek2017}, this approach can tackle sufficiently large instances and is increasingly popular. 


In this paper, we aim to develop a method combining both approaches, i.e., utilize an efficient low-rank-promoting penalty in the Burer-Monteiro approach. We present efficient first-order numerical algorithms for solving the resulting penalized problem, with (almost)
 linear-time per-iteration complexity. This makes the combined approach applicable to a wide range of practical problems. 

In a case study, we focus on certain combinatorial optimization problems and inference in graphical models. 
We show that despite the non-convexity of the penalized problem, our approach successfully recovers rank-one solutions in practice.
We compare our solutions against  non-penalized SDP, belief propagation, and state-of-the-art branch-and-bound solvers of~\cite{krislock2014improved}.

\paragraph{Contribution.} Our contributions can be summarized as follows. We show
\begin{enumerate}
\item  convergence properties of optimization methods employing a wide class of penalty functions that promote low-rank solutions; 
\item linear-time algorithms for computing the gradient of these penalty functions; 
\item computational results on the penalized SDP relaxation of maximum a posteriori (MAP) estimates in Markov random fields (MRF), which considerably improve upon the results obtained by interior-point methods and randomized rounding.
\end{enumerate}

This allows for both well-performing and easy-to-analyze low-rank methods for SDPs coming from graphical models, combinatorial optimization, and machine learning. 

\paragraph{Paper structure.} This paper is organized as follows. First, we define the conic optimization problem together with a penalized form with a list of suitable penalization functions. Next, we present theoretical guarantees for solution recovery. These extend known results for solution recovery to the penalty case. Then, we consider the MAP problem in Markov Random Fields (MRF) and introduce an iterative procedure for it, together with a first-order method for solving a subproblem at each step; we also show how to compute the gradients efficiently.  Finally, we provide computational experiments for different inference problems in MRF.

\section{Background}\label{sec:setup}

SDP is the following conic optimization problem:
\begin{align}
\min_{X\in \bS^n_{+}} \ & 
  \sum_{i \in I}
   f_i( \tr XS_i ) 
   \tag{SDP}
   \label{eq:sdp}
   \\ 
  \text{s.t.:\;} & g_j( \tr X C_j ) \leq 0, \qquad j \in J, 
\notag
\end{align}
where 
$X\in \bS^n_{+}$ denotes that the $n \times n$ matrix variable $W$ is symmetric positive semidefinite,
$I$ and $J$ are finite index sets, 
each $f_i$ and $g_j$ are convex functions $\R^{n\times n} \to \R$, and $C_j \in R^{n\times n}$  and $S_i\in \R^{n\times n}$  are constant matrices.

In the context of combinatorial optimization, one may also consider even more powerful methods such as Sum-of-Squares hierarchies of \cite{Parrilo}.
However, even an SDP relaxation, which is, in fact, the first step of this hierarchy, may be too computationally challenging. It is usually solved by interior-point methods in the dimension that is quadratic in the number of variables and thus becomes intractable even for medium-scale problems (with a few thousand variables). The problem becomes even less scalable for higher orders of the hierarchy since it requires one to solve SDP with $n^{\Theta (d)}$ variables, where $d$ is the level of the hierarchy.

\subsection{Low-rank Relaxations and Penalized Problem}
First, let us formally define our notion of a penalty function and explain related work on first-order
methods for SDP.

\begin{assumption}
Eq.~\eqref{eq:sdp} has an optimum solution with rank~$r$.
\end{assumption}

Let us consider the following proxy problem:
\begin{align}
P_{q,\lambda} \doteq & \min_{V\in R^{n\times q}}
  \sum_{i \in I}
   f_i( \tr(V^\top S_iV) )
   +  R_{q,\lambda}(V) 
   \label{eq:prox} 
   \tag{P-SDP}\\
& \;\text{s.t.: } g_j( \tr(V^\top C_j V) ) \leq 0, \; j \in J.
\nonumber
\end{align}
Where $q > r$ and $R_{q,\lambda}(V)$ satisfies the following:

\begin{definition}[Strict penalty function]\label{def:prop}
A function $\mathcal{R}_{q,\lambda}(V): \R^{n \times n} \to \R$ 
is a \emph{penalty function} that promotes low-rank solutions if for some integers $q' \le q$ and a multiplier $\lambda \in \R^+$: 
\begin{align*}
\lim_{\lambda \to \infty} \mathcal{R}_{q,\lambda}(V)
= \begin{cases}
\; \textrm{0} & \; \mbox{ if } \rank(V) < q',\\
\; \infty & \; \mbox{ if } \rank(V) = q.
\end{cases}
\end{align*}
Moreover, if $q' = q$ and the 
$\rank(X) < q$, then $\mathcal{R}_{q,\lambda}(V) = 0, \ \forall \lambda$, $R(V)$ is a \emph{strict penalty function}.
\end{definition}

We use the word {\it penalty} instead of penalty function that promotes low-rank solutions, where there is no risk of confusion. This notion of a penalty is rather wide. When multiplied by $\lambda$, a determinant is a prime example. One may also consider functions of the following quasi-norms.
\begin{enumerate}
\item The nuclear norm: 
\begin{equation}
\|X\|_* = \sum_i\sigma_i,
\label{eq::trace_norm}
\end{equation}
where $\sigma_i$ is the $i$-th singular value, cf. \cite{Lemon}. 
The norm is also known as a trace norm, Schatten 1-norm, and Ky Fan norm. 
As shown by \cite{Srebro}, in the method of \cite{BurerMonteiro}, one can benefit from a bi-Frobenius reformulation:
\begin{align*}
\|X\|_{*} & = \min_{\substack{U\in\mathbb{R}^{n\!\times\! d} \\ V\in\mathbb{R}^{n\!\times\! d}:X=UV^{T}}}\|U\|_{F}\|V\|_{F} 
\\ & = \min_{U,V:X=UV^{T}}\frac{\|U\|^{2}_{F}+\|V\|^{2}_{F}}{2}.
\end{align*}
There are also truncated \cite{HuTruncated} and capped variants \cite{SunTruncated}.
\item Schatten-${p}$ quasi-norm for $p > 0$:
\begin{align}
\|X\|_{S_{p}}=\left(\sum^{n}_{i=1}\sigma^{p}_{i}(X)\right)^{1/p},
\end{align}
where $\sigma_{i}(X)$ denotes the $i$-th singular value of $X$.

\item A smoothed variant of Schatten-${p}$ quasi-norm by \cite{Pogodin} for $p, \epsilon > 0$:
\begin{equation}
\schatteneps{X}{p}^p = \sum\limits_{i =  1}^{n} \left(\sigma_i^2 + \varepsilon\right)^{\frac{p}{2}} = \tr\left(X^\top X+\varepsilon I\right)^{\frac{p}{2}}.
\label{eq::schatten_norm_smoothed}
\end{equation}

\item Tri-trace quasi-norm of \cite{Shang2018}:
\begin{align}
\|X\|_{\textup{Tri-tr}}
=
\min_{X=UV\Upsilon^\top}\|U\|_{*}\|V\|_{*}\|\Upsilon\|_{*},
\end{align}
which is also the Schatten-1/3 quasi-norm.
\item Bi-nuclear (BiN) quasi-norm of \cite{Shang2018}:
\begin{align}
\|X\|_{\textup{BiN}}
=
\min_{X=UV^{T}}\|U\|_{*}\|V\|_{*},
\end{align}
which is also the Schatten-${1/2}$ quasi-norm.
\item Frobenius/nuclear quasi-norm of \cite{Shang2018}:
\begin{align}
\|X\|_{\textup{F/N}} =
\min_{X=UV^{T}}\|U\|_{*}\|V\|_{F},
\end{align}
which is also the Schatten-${2/3}$ quasi-norm.
\end{enumerate}

We also note there has been considerable interest in the analysis of low-rank approaches without penalization, especially in matrix-completion applications. Much of the analysis goes back to the work of \cite{keshavan2010matrix}. For further important contributions, see \cite{arora2012computing}. 

\subsection{Entropy Viewpoint}

One could see the penalty functions introduced above from the entropy-penalization perspective. 
This is useful not only from a methodological standpoint, but also from a computational one. 
To this end, we consider the Tsallis entropy:
\[
\ent^T_\alpha(X) = \frac{1}{1-\alpha} \left( \frac{\tr{X^\alpha}}{(\tr{X})^\alpha} - 1\right). 
\]

The Tsallis entropy is crucial in our study because it generalizes many popular penalties considered earlier. 
The Schatten $p$-norm 
coincides with the Tsallis entropy $\ent^T_p$ over a set of matrices with a fixed trace norm, so that the tri-quasi norm and bi-nuclear norm (2--6) are covered as well. 
The Log-Det function, $- \log \det X$, which is also used in low-rank SDP, 
is up to an additive constant factor relative (Shannon) entropy taken concerning a unit matrix, 
while Renyi ($\ent^R$) and von Neumann ($\ent^N$) entropies, 
\begin{align*}
& \ent^R_\alpha(X) = \frac{\log \tr (X/\tr X)^\alpha}{1-\alpha} \text{ and }\\
& S_\alpha^N(X) = - \tr(\log(X/\tr X) \cdot X/\tr(X)), 
\end{align*}
respectively, can also be used as penalties to promote a low-rank solution. To the best of our knowledge, neither Renyi, von Neumann, nor Tsallis entropies have been studied in the context of low-rank SDP. 

\section{Exact Recovery}\label{sec:theory}

Let us now present a unified view of the penalties and their properties:
\begin{lemma}
Any of:
\begin{enumerate}
\item $\lambda \det ( X )$;
\item $\lambda \sigma_{q}(X)$, where $\sigma_{i}(X)$ denotes the $i$-th singular value of $X$;
\item Tsallis, Renyi, and von Neumann entropies defined on the last $n-q+1$ singular values;
\item $\lambda \max \left\{ 0, \frac { \|X\|_{*} }{ \max \{ \sigma_{\min}(X), \sigma_{q}(X) \} } - q \right\},$
\end{enumerate}
is a penalty function that promotes low-rank solutions.
Moreover, penalties 1--3 are strict.
\label{lem1}
\end{lemma}
\begin{proof} \emph{Sketch}.
(1.) The proof is by simple algebra.
(2.) If $\sigma_{q}(X)$ is 0, we know the rank is $q - 1$ or less. Otherwise, for large values of $\lambda$, the value of the penalties goes to infinity, and hence $q' = q$.
(3.) The definition of entropy assumes that $S(0, ..., 0) = 0$, thus all entropies are strict penalty functions by definition. 
(4.) First, consider the case where all non-zero singular values are equal. In that case, $ \|X\|_{*} / \sigma_{\min}(X) = \rank(X)$, and subtracting $q$ results either in a non-positive number when the rank is less than $q$ or a positive number otherwise.
If the singular values are non-equal, $ \|X\|_{*} / \sigma_{\min}(X) $ provides an upper bound on the rank of $X$, which can be improved as suggested. The use of the upper bound results in the value of the penalty tending to infinity for values between $q'$ and $q$ in the large limit of $\lambda$. 
\end{proof}

Crucially, under mild assumptions, any penalty allows for the recovery of the optimum of a feasible instance of \eqref{eq:sdp} from the iterations of an algorithm on the non-convex problem in variable $V \in \R^{n \times r}$, such as in the methods of \cite{lowrankMAP} or \cite{BurerMonteiro}. In contrast to the traditional results of \cite{burer2005local}, \cmmnt{,burer2005local} who consider the $\det$ penalty, we allow for the use of any strict penalty function. 

\begin{theorem}
Assume that we solve the proxy problem \eqref{eq:prox} iteratively and $\mathcal{R}_{q,\lambda}(V)$ is a strict penalty function that promotes low-rank solutions. In each iteration, if $\mathcal{R}_{q,\lambda}(V) \neq 0$, we increase $\lambda $ (e.g., set $\lambda_{t+1} = \gamma \lambda_t$, with $u > 1$ as some fixed parameter). Furthermore, let us assume that the solution we found is denoted by $\tilde V_{q}$ with $\rank(\tilde V_{q})= q' < q$. Let us also denote $\tilde V_{q'} \in \R^{n \times q'}$ some factorization of $\tilde V_{q} \tilde V_{q}^\top$ (such factorization exists because $\rank(\tilde V_{q})= q'$). Also assume that we have an optimal solution of \eqref{eq:sdp}, $X^*$ with a~$\rank(X^*) = r$.

If 
\begin{equation}
V_{q'+r+1} \triangleq [ \tilde V_{q'}, {\bf 0}_{n\times r}, {\bf 0}_{n\times 1}]
\end{equation}
is a local minimum of $P_{q'+r+1,\lambda}$, then 
$(\tilde V_{q'}) \tilde V_{q'}^\top$ is a global solution of \ref{eq:sdp}.
\label{strict}
\end{theorem}

\begin{proof}
Let us define a family of matrices for $\tau \in [0,1]$ as follows:
\begin{align*}
V(\tau) \triangleq [ \sqrt{\tau} \tilde V_{q'}, \sqrt{(1-\tau)} V_* , {\bf 0}_{n\times 1}],
\end{align*}
where $ (V_*)^\top (V_*)$ is some factorization of $X^*$ with $V_* \in \R^{n \times r} $.

Note that $\forall \tau$, we have $\rank(X(\tau)) < r+q'+1$, and hence
$\forall \lambda, \tau:
R_{q'+r+1, \lambda}(V(\tau)) = 0$.
Now, assume the contradiction, that is, $V_{q'+r+1}$ is a local optimum solution but $\tilde V_{q'}$ is not a global solution.

We show that $\forall \tau \in [0,1]$, $V(\tau)$ is a feasible solution. Indeed, for any $j \in J$ we have
\begin{align*}
    & g_j(\tr(V(\tau)^T C_j \tr(V(\tau)) \leq \\
    & \quad \tau g_j(\tr( [ \tilde V_{q'}, {\bf 0}_{n\times r+1}]^\top C_j [ \tilde V_{q'}, {\bf 0}_{n\times r+1}])) + \\
    & \quad(1-\tau) \tr( [  {\bf 0}_{n\times q'}, V_*,{\bf 0}_{n\times 1}]^\top C_j [  {\bf 0}_{n\times q'}, V_*,{\bf 0}_{n\times 1}])) = \\
    & \hspace{20mm} \tau g_j(\tr(  \tilde V_{q'} ^\top C_j   \tilde V_{q'} )) + (1-\tau) \tr(  X^* C_j  )) \leq  0.
\end{align*}

We just showed that for each $V(\tau), \tau \in [0,1]$ is a feasible point. Now, let us compute the objective value at this point. 
For all $ \tau \in [0,1]$, we have
\[
  \sum_{i \in I}  \tr( V(\tau)^\top S_i V(\tau)) \leq  \tau \sum_{i \in I}
    \tr( \tilde V_{q'}^\top S_i \tilde V_{q'} ) 
   +
\]
\[
 +
 (1-\tau)  
 \sum_{i \in I} \tr( X^* S_i  ) <  \sum_{i \in I} \tr( \tilde V_{q'}^\top S_i \tilde V_{q'} ),
\]
which is a contradiction under the assumption that 
$\tilde V_q$ is a local optimum.
\end{proof}

\section{Efficient Implementation for MAP Inference
}\label{sec:algo}

A pairwise Markov Random Field (MRF) is defined for an arbitrary graph $G = (V, E)$ with $n$ vertices. We associate a binary variable $x_i\in \{-1, +1\}$ with each vertex $i\in V$. Let $\theta_i: \{\pm 1\}\to \bR$ and  $\theta_{ij}: \{\pm1\}^2 \to \bR$ defined for each vertex and edge of the graph be vertex and pairwise potential, respectively. Thus, \emph{a posteriori} distribution of $x$ follows the Gibbs distribution: 
\[
p(x|\theta) = \frac{1}{Z(\theta)} e^{U(x|\theta)},
\]
with $U(x;\theta) = \sum_{i\in V} \theta_i(x_i) + \sum_{(i,j)\in E} \theta_{ij}(x_i, x_j).$
The maximum \emph{a posteriori} (MAP) estimate is then 
\[
\hat{x} = \argmax\limits_{x\in \{-1, 1\}^n} p(x|\theta) = \argmax\limits_{x\in \{-1, 1\}^n} U(x;\theta),
\tag{MAP}
\]
which is its turn an NP-hard binary quadratic optimization problem, 
\[
\hat x = \argmax\limits_{x\in \{-1, 1\}^n} x^\top S x, 
\]
with indefinite matrix $S$. The SDP relaxation for this problem is given by \cite{goemans1995improved,Nesterov}:
\begin{gather}
 \min_{X \in \bS^+_n} \tr SX, \quad \text{s.t.: } X_{ii} = 1,\label{eq:lin_sdp} 
\end{gather}
which also covers the Ising model in statistical physics and a number of combinatorial optimization problems. We believe that the approach can be extended to a general setup given by Eq.~\eqref{eq:sdp}. 

An entropy-penalized SDP relaxation of \eqref{eq:lin_sdp}  has the form 
\begin{align}
 \min_{X \in \bS^+_n} & \tr V^\top S V + R_{\lambda}(V), \quad \text{s.t.: } \|V^i\|_2^2 = 1,\label{eq:lin_sdp_rel}\tag{EP-SDP}
\end{align}
where $V^i$ is the $i$-th column of matrix $V \in \bR^{n\times k}$, $X = V V^\top$. 

\subsection{Numerical Method. } To solve Problem~\eqref{eq:lin_sdp_rel}, we use the Augmented Lagrangian method starting from a sufficiently small value of the penalty parameter $\lambda > 0$ and increasing it in geometric progression, $\lambda_{k+1} = \lambda_{k} \gamma$, with $\gamma > 1$, as summarized in Algorithm \ref{algo:01}. 
The efficiency of the method is due to the efficient computability of gradients of Tsallis, Renyi, and von~Neumann entropies:

\begin{algorithm}[]
 \SetAlgoLined
 \label{algo:01}
    \KwData{Quadratic matrix $S$ of the MAP inference problem, staring point $\lambda_0$, $\gamma > 1$, step size policy $\{\eta_k\}_{k\ge 1}$ accuracy parameters $\varepsilon$, $\epsilon$}
    \KwResult{Solution $V_*$ as a local minimum of \eqref{eq:lin_sdp_rel} of unit rank}
    \Begin{
        $V_0 \leftarrow$ random initialization in $\R^{n\times k}$\;
        \While{$\tr (V_t^\top V_t) - \lambda_{\max} V_t > \varepsilon$}{
        Find local minimum of \ref{eq:lin_sdp_rel}\!$(S, \lambda_t)$ starting from $V_{t-1}$, assign it to $V_t$\;
        \While{$\nabla (\tr V^\top S V + R_{\lambda}(V)) \le \epsilon$}{
        $V = V - \eta_k \nabla (\tr V^\top S V + R_{\lambda}(V))/\|\nabla (\tr V^\top S V + R_{\lambda}(V))\|_2$\;
        $V_i \leftarrow V_i/\|V_i\|_2$ for each row $V_i$\;
        }
        $\lambda_{t+1} = \lambda_t\cdot \gamma$\;
        }
    }
    {\bf Return:} first singular vector of $V_t$.\;
 \caption{Entropy-Penalized SDP.}%
\end{algorithm}

\begin{lemma}
For any matrix $V \in \bR^{n\times k}$ with $k = {\mathcal O}(1)$, let $X(V) = V^\top V$. Then, gradients of $\ent_\alpha^T(X)$, $\ent^R(X)$, and $\ent^N(X)$ can be computed in $\mathcal{O}(n)$ time. Moreover, if the number of non-zero elements in matrix $A$ is $\mathcal{O}(n)$, then the iteration complexity of Algorithm \ref{algo:01} is $\mathcal{O}(n)$. 
\end{lemma}
\begin{proof}
We start our analysis with Tsallis entropy. First, compute the gradient of $\ent^T_\alpha$ in $V$:
\[
\frac{\partial \ent_\alpha^T(X)}{\partial V} = \frac{\alpha}{1-\alpha}\left(\frac{X^{\alpha-1}}{(\tr{X})^{\alpha}} - \frac{\tr{X^\alpha} }{(\tr{X})^{\alpha + 1}} I \right)V.
\] 

Similarly for Renyi, $\ent^R(X)$, and von Neumann, $\ent^N(X)$, entropies we have
\[
\frac{\partial \ent_\alpha^R(X)}{\partial V} = \frac{\alpha}{1-\alpha}\frac{X^{\alpha-1}}{\tr{X^{\alpha}}} \left( I - \frac{X}{\tr{X}} \right) V 
\]
and
\[
\frac{\partial \ent_\alpha^N(X)}{\partial V} = \frac{\tr{X} I - X}{(\tr{X})^2} \left( I + \log{\frac{X}{\tr{X}}} \right)V.
\]
Following \cite{Holmes}, the singular-value decomposition of matrix $V = U_1 D U_2$ with $U_1 \in \R^{n\times n}$, $D\in \R^{n\times k}$, and $U_2 \in \R^{k\times k}$ can be performed in $\mathcal O(\min{(nk^2, n^2k)}) = \mathcal O(nk^2)$ time. 

For any $\alpha > 1$, the product $X^{\alpha-1}\cdot V = U_1 D^{2\alpha - 1} U_2$ can be computed in time $\mathcal O(n)$ together with $\tr{X^\alpha} = \tr{D^{2\alpha}}$ and $\tr{X} = \tr{D^2}$. Thus, for a fixed $k$,  the gradient $\frac{\partial \ent^T(X)}{\partial V}$ computation time is linear in its dimension. (Here, for any $\alpha \in (0, 1)$, we use the identity $\partial \lambda_i = {\bf  v_i}^\top \partial X {\bf v_i}$.) To finish the proof of the statement, it remains to note that matrix-vector multiplication takes $\mathcal{O}(n)$ time for any matrix with $\mathcal{O}(n)$ non-zero entries. 
\end{proof}

\section{Case Study}\label{sec:case_study}

In this section, we compare our penalized algorithm with other conventional approaches to MAP problems. We fix the width of factorization to $k = 10$, since there is no significant gain in practice for larger values of $k$, cf.  \cite{Montanari}. We choose $2\eta_k \beta = 1$, where $\beta$ is the Lipschitz constant of the gradient in $\ell_2$ norm and $\gamma = 3/2$. Parameters $\lambda_0$ and $\gamma$ of Algorithm~\ref{algo:01} are usually chosen by a few iterations of random search. It is usually enough to have about 35 iterations for penalty updates and a few hundred iterations to find a local minimum using Algorithm~\ref{algo:01}. We emphasize that matrices we obtain by solving \ref{eq:lin_sdp_rel} are rank-one solutions on all MAP instances presented. Thus, we do not need any further rounding procedure.

First, in Table~1, we show the performance of our algorithm on selected hard MAP inference problems from the BiqMac collection\footnote{http://biqmac.uni-klu.ac.at/biqmaclib.html}. We selected a few of the hardest instances ("gkaif" among them)---dense quadratic binary problems of 500 variables. 

\begin{table}[!t]
\label{tbl:new}
\centering
\begin{adjustbox}{width=\columnwidth,center}
\begin{tabular}{l|c|c|c|c|c}
Instance & gka1f  & gka2f & gka3f & gka4f & gka5f\\
\hline\hline
\multicolumn{6}{c}{SDP}\\
\hline\hline 
objective &  59426 &  97809   &   1347603 & 168616 & 185090  \\
upper bound &66783 & 109826 & 152758 & out of & out of \\
time [s] & 669 & 673 &592 & memory & memory\\
\hline\hline
\multicolumn{6}{c}{EP-SDP}\\
\hline\hline
objective & 60840 &  {\bf 99268}   &  {\bf 136567} & {\bf 170669} & {\bf 189762}  \\
upper bound & n/a & n/a & n/a & n/a & n/a \\
time [s] &   3.3 &  5.0  &  5.3 & 5.2 & 5.7  \\ 
\hline\hline
\multicolumn{6}{c}{Gurobi}\\
\hline\hline
objective &  {\bf 64678} &  97594  & 131898 & 162875 & 189324  \\
upper bound &   73267 &  112223  &  153726 & 190073 & 218428 \\ 
time [s]. &   70 &  70  & 71 & 70 & 70  \\
\end{tabular}
\end{adjustbox}
\caption{Results for the BiqMac collection.}
\end{table}

We compared our algorithm (EP-SDP with Tsallis entropy and $\alpha=2$) with the plain-vanilla semidefinite programming instance solved by the interior-point method, possibly with rounding using the best of one thousand roundings of \cite{goemans1995improved} and also with Gurobi solver for Mixed-Integer Problems. To avoid any confusion, we solve the corresponding maximization problems; by the objective value, we mean the value at a feasible solution produced by the method (e.g., rounded solution of SDP relaxation), which is a lower bound for the corresponding problem.
Because these problems are of the same size (but varying density), the running time of each method is almost constant. It took around 10 minutes for CVXPY to solve the SDP relaxation, and it runs out of memory for the two problems with higher density. Within five seconds, EP-SDP obtains results that are better than what Gurobi can produce in 70 seconds. 

\begin{table}[!th]
\centering
\begin{tabular}{l|c|c|c|c}
& \multicolumn{4}{c}{GSET Instance}\\
\hline\hline
& 1 & 2 & 3& 4 \\
\hline
EP-SDP\\(T, $\alpha=2.0$)\!\!& 11485 & 11469 & 11429 & 11442  \\ 
(T, $\alpha=1.1$)\!\!& 11454 & 11463 & 11444 &  11508  \\ 
(R, $\alpha=5$)\!& 11508 & {\bf 11519} & 11496 & {\bf 11531}   \\ 
(R, $\alpha=10$) & {\bf 11520} & 11420 & 11523 & 11523    \\
SDP & 11372 & 11363 & 11279 & 11355 \\ 
Loopy BP & 10210 & 10687 & 10415 & 10389 \\
Mean-Field & 11493 & 11515 & {\bf 11525} & 11512\\
\hline\hline
& \multicolumn{4}{c}{GSET Instance}\\
\hline\hline
& 5 & 6 & 7& 8 \\
\hline
EP-SDP\\(T, $\alpha=2$)& 11427 & 2059 & 1888 & 1866\\ 
(T, $\alpha=1.1$)  & 11506 & 2075 & 1858 &  1895\\ 
(R, $\alpha=5$)  & 11527 & {\bf 2127} & {\bf 1942}  & 1954\\ 
(R, $\alpha=10$) & {\bf 11538} & 2112 & 1940  & {\bf 1958}\\
SDP  & 11313 &  1945 &  1728  & 1727\\ 
Loopy BP & 10143 & 1076 & 964 &  731\\
Mean-Field  & 11528 & 2096 & 1906 & 1912\\
\hline\hline
\end{tabular}
\begin{tabular}{l|c|c|c|c|c}
& \multicolumn{5}{c}{GSET Instance}\\
\hline\hline
& 9 & 10& 11& 12 & 13 \\
\hline
EP-SDP\\(T, $\alpha=2$) & 1933 & 1882 & 532 &  530 & 560 \\ 
(T, $\alpha=1.1$)   & 1969 & 1861 & 544 &  536 & {\bf 568} \\ 
(R, $\alpha=5$)  & 1992 & 1960 & {\bf 550} & {\bf 548} & {\bf 568} \\
(R, $\alpha=10$)  & {\bf 2006} & {\bf 1982}  & 544 & 546 & 564 \\
SDP  & 1767 & 1784 & 524  & 514  &  540 \\ 
Loopy BP & 1021 & 820 & 424 & 412 & 482 \\ 
Mean-Field  & 1940 & 1902 & 542 & 538 & 564\\
\end{tabular}
\label{tbl:02}
\caption{Results for the GSET collection.}
\end{table}

In  our study, parameter $\alpha$ of entropies $\mathcal{S}^T_\alpha$, $\mathcal{S}^N_\alpha$, and $\mathcal{S}^R_\alpha$ is chosen on an exponential grid from 1 to 10 with a step 1.1. After experimentation, we note that $\alpha = 1.1$ and $\alpha = 5.0$ seem to improve the results the best for the low-rank SDP with Tsallis and Renyi entropies, respectively, although the difference between different $\alpha \in (1,10)$ is not very significant for either of the (Tsallis and Renyi) entropies. 

Table~2 summarizes the results of solving the Max-Cut problem over a GSET collection of sparse graphs\footnote{https://sparse.tamu.edu/Gset}. As we see from the experiments, the results of applying suitable entropy often  outperform both the plain-vanilla SDP with the classical Goermans-Williamson rounding, the mean-field approximation, as well as the results of UGM solver\footnote{https://www.cs.ubc.ca/~schmidtm/Software/UGM.html} for loopy belief propagation and mean-field inference. 
It is worth noting that for several instances of the GSET graph collection, loopy belief propagation provides rather weak results. 
Usually, strong results of the loopy belief propagation are complementary to those of the mean-field approximation, which is supported by our empirical results. Results of both loopy belief propagation and mean-field approximation can be substantially improved using the linear-programming belief-propagation approach (LP-BP). 

\begin{figure}[t!]
\centering
      \begin{tikzpicture}[scale=0.73]
    \begin{axis}[legend pos=north east, title={},xlabel=power of $\gamma$, ylabel=second singular value of 
    $V$]
         \addplot table [x=iter, y=lambda2-g2, col sep=comma] {decreasing-rank.txt};
	  \addlegendentry{$\alpha=2$ and $\gamma = 2$}
         \addplot table [x=iter, y=g1.5, col sep=comma] {decreasing-rank.txt};
         \addlegendentry{\hspace{2.5mm}$\alpha=2$ and $\gamma = 1.5$}
         \addplot table [x=iter, y=alpha5-g2, col sep=comma] {decreasing-rank.txt};
         \addlegendentry{$\alpha=5$ and $\gamma = 2$}
    \end{axis}
    \end{tikzpicture}
\caption{Rank decrement.}\label{rankdef}
\end{figure}
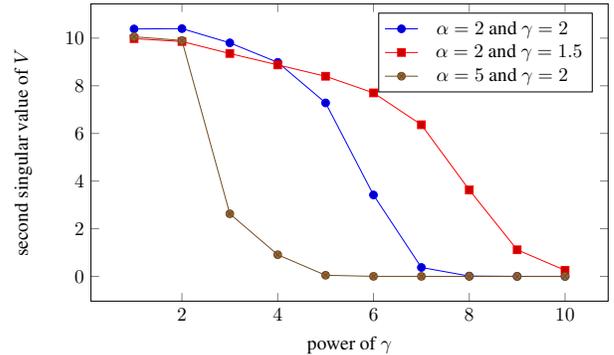

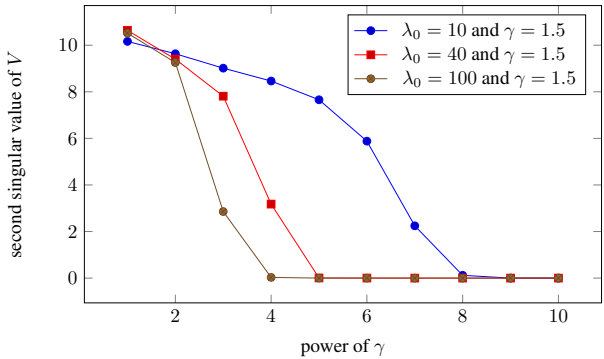
\begin{figure}[t!]
\centering
      \begin{tikzpicture}[scale=0.73]
    \begin{axis}[legend pos=north east, title={},xlabel=power of $\gamma$, ylabel=second singular value of 
     $V$]
         \addplot table [x=iter, y=10, col sep=comma] {inc-lambda.txt};
	  \addlegendentry{$\lambda_0=10$ and $\gamma = 1.5$}
         \addplot table [x=iter, y=40, col sep=comma] {inc-lambda.txt};
         \addlegendentry{$\lambda_0=40$ and $\gamma = 1.5$}
         \addplot table [x=iter, y=100, col sep=comma] {inc-lambda.txt};
         \addlegendentry{\hspace{2mm}$\lambda_0=100$ and $\gamma = 1.5$}
    \end{axis}
    \end{tikzpicture}
\caption{Rank decrement with multistart.}\label{rankdef2}
\end{figure}

We also want to point out that our iterative algorithm successfully decreases the rank of the solution. The higher the penalization parameter, the lower the rank. We illustrate this in Figure~\ref{rankdef}, where for Tsallis entropies with $\alpha = 2$ and $\alpha = 5$, we plot the second singular value of  matrix $V$. For this plot, we considered the Max-Cut problem for the first graph from the GSET collection and the Tsallis entropy as the penalization function. In Figure~\ref{rankdef2}, we illustrate the same concept for fixed penalization (Tsallis entropy with $\alpha = 2$) and different initial values of the multiplier $\lambda$. We observe that for different penalization functions and update schemes, the rank of the solution decreases gradually with each step. In practice, our iterative algorithm could be seen as a universal rounding procedure for SDP relaxations. Indeed, if we choose a large-enough penalization update (e.g., $\gamma = 2$ as in Figure 1), we easily obtain a rank-one solution that is not worse and often is substantially better than solutions obtained by randomized rounding.

Overall, we would like to stress that Algorithm~\ref{algo:01} is very fast. This is shown in Figure~3 and Table~3, where we compare run times of EP-SDP, low-rank Burer-Monteiro approach (LR-SDP), and interior-point method solvers (SDP) for various Erdos-Renyi random graphs. From the data, we see that (assuming the fixed width of factorization $k= 10$) EP-SDP run time increases linearly with the number of vertices. Indeed, throughout the benchmark instances tested, the run time does not exceed a few seconds per each of the test cases. At the same time, the bound is often  almost as good as that of the Branch and Bound Biq-Mac Solver of \cite{krislock2014improved},  which requires a significant amount of time. 

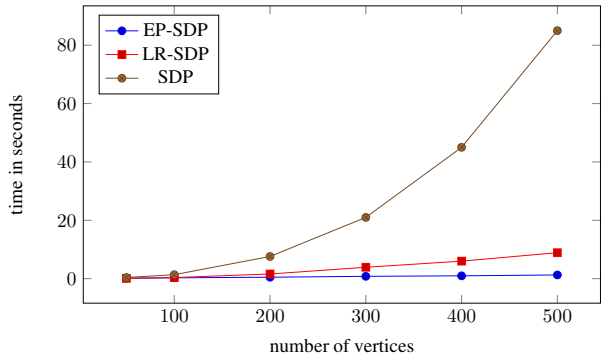
\begin{figure}[t!]
\centering
      \begin{tikzpicture}[scale=0.73]
    \begin{axis}[legend pos=north west, title={},xlabel=number of vertices, ylabel=time in seconds]
         \addplot table [x=vert, y=EPSDP, col sep=comma] {time.txt};
	  \addlegendentry{EP-SDP}
         \addplot table [x=vert, y=SDPLR, col sep=comma] {time.txt};
         \addlegendentry{LR-SDP}
         \addplot table [x=vert, y=SDP, col sep=comma] {time.txt};
         \addlegendentry{SDP}
    \end{axis}
    \end{tikzpicture} 
\caption{Time complexity}\label{time}
\end{figure}

\begin{table}[!t]
\label{tbl:time}
\centering
\begin{tabular}{l|c|c|c}
Instance &  {\bf EP-SDP} & {\bf LR-SDP} & {\bf SDP}\\
\hline\hline
E-R(50, 0.2) & 0.2s &  0.1s   &   0.4s  \\
\hline
G(100, 0.2) & 0.3s &  0.4s   &  1.4s  \\
\hline
G(200, 0.2) &  0.5s &  1.6s   &   7.6s  \\
\hline
G(300, 0.2) &  0.8s &  3.9s   &   21.0s  \\
\hline
G(400, 0.2) &  1.0s &  6.0s   &   45.0s  \\
\hline
G(500, 0.2) &  1.3s &  8.9s   &   85.0s  \\
\end{tabular}
\caption{Run time for random graphs.}
\end{table}

\section{Conclusions}\label{sec:conclusion}

This paper presented a unified view of the penalty functions used in low-rank semidefinite programming using entropy as a penalty. 
This makes it possible to find a low-rank optimum, where there are optima of multiple ranks.
Semidefinite programs with an entropy penalty can be solved efficiently using first-order optimization methods with linear-time per-iteration complexity, which makes them applicable to large-scale problems that appear in machine learning and polynomial optimization. 
Our case study illustrated the practical efficiency on binary MAP inference problems. 
The next step in this direction is to consider the structure of the SDP, which seems to be  crucial for further scalability. 

\section*{Acknowledgements} 
The work of Martin Tak{\' a}{\v c} was partially supported by the U.S. National Science Foundation, under award numbers NSF:CCF:1618717, NSF:CMMI:1663256, and NSF:CCF:1740796.
The work at LANL was supported by the U.S. Department of Energy through the Los Alamos National Laboratory as part of the GMLC ``Emergency Monitoring and Control through new Technologies and Analytics'' project and multiple LANL LDRD projects. Los Alamos National Laboratory is operated by Triad National Security, LLC, for the National Nuclear Security Administration of U.S. Department of Energy (Contract No. 89233218CNA000001).

\bibliographystyle{named}
\bibliography{biblio}

\end{document}